\documentclass{amsproc}
\usepackage[utf8]{inputenc}
\usepackage{amsmath}
\usepackage{amsfonts}
\usepackage{amssymb}
\usepackage{amsthm}
\usepackage{mathrsfs}
\usepackage{subfigure}
\usepackage[usenames,dvipsnames]{xcolor}
\usepackage{pstricks}
\usepackage{graphicx}
\usepackage[all]{xy}

\newtheorem{theorem}{Theorem}[section]
\newtheorem{lemma}[theorem]{Lemma}
\newtheorem{coro}[theorem]{Corollary}
\newtheorem{prop}[theorem]{Proposition}

\theoremstyle{definition}

\theoremstyle{remark}
\newtheorem{remark}[theorem]{Remark}

\numberwithin{equation}{section}

\newcommand{\Op}{\operatorname{Op}}

\newcommand{\nwc}{\newcommand}
\nwc{\eps}{\epsilon}
\nwc{\ep}{\epsilon}
\nwc{\vareps}{\varepsilon}
\nwc{\Oph}{\operatorname{Op}_\hbar}
\nwc{\la}{\langle}
\nwc{\ra}{\rangle}

\nwc{\mf}{\mathbf} 
\nwc{\blds}{\boldsymbol} 
\nwc{\ml}{\mathcal} 

\nwc{\defeq}{\stackrel{\rm{def}}{=}}

\nwc{\cE}{\ml{E}}
\nwc{\cN}{\ml{N}}
\nwc{\cO}{\ml{O}}
\nwc{\cP}{\ml{P}}
\nwc{\cU}{\ml{U}}
\nwc{\cV}{\ml{V}}
\nwc{\cW}{\ml{W}}
\nwc{\tU}{\widetilde{U}}
\nwc{\IN}{\mathbb{N}}
\nwc{\IR}{\mathbb{R}}
\nwc{\IZ}{\mathbb{Z}}
\nwc{\IC}{\mathbb{C}}
\nwc{\IT}{\mathbb{T}}
\nwc{\IS}{\mathbb{S}}
\nwc{\tP}{\widetilde{P}}
\nwc{\tPi}{\widetilde{\Pi}}
\nwc{\tV}{\widetilde{V}}
\nwc{\supp}{\operatorname{supp}}
\nwc{\rest}{\restriction}
\let \d \relax
\nwc{\d}{\partial}
\nwc{\Cor}{\mathscr{C}}

\nwc{\todo}[1]{$\clubsuit$ {\tt #1}}

\begin{document}

\title[Asymptotic regularity of sub-Riemannian eigenfunctions]{Asymptotic regularity of sub-Riemannian eigenfunctions in dimension $3$: the periodic case}

\author{Gabriel Rivi\`ere}
\address{Laboratoire de Math\'ematiques Jean Leray, Nantes Universit\'e, UMR CNRS 6629, 2 rue de la Houssini\`ere, 44322 Nantes Cedex 03, France}

\address{Institut Universitaire de France, Paris, France}

\email{gabriel.riviere@univ-nantes.fr}

\thanks{This work benefited from the support of the Institut Universitaire de France, of the Centre Henri Lebesgue (ANR-11-LABX-0020-01) and of the PRC grant ADYCT (ANR-20-CE40-0017) from the Agence Nationale de la Recherche.}

\date{October, 28th 2023.}

\dedicatory{Dedicated to the memory of Steve Zelditch}


\begin{abstract}
On the unit tangent bundle of a compact Riemannian surface of constant nonzero curvature, we study semiclassical Schr\"odinger operators associated with the natural sub-Riemannian Laplacian built along the horizontal bundle. In that set-up, the involved Reeb flow is periodic and we show that high-frequency Schr\"odinger eigenfunctions enjoy extra regularity properties. As an application, we derive regularity properties for low-energy eigenmodes of semiclassical magnetic Schr\"odinger operators on the underlying surface by considering joint eigenfunctions with the Reeb vector field.
\end{abstract}

\maketitle

\section{Introduction}

Let $(M,g)$ be a smooth, compact, oriented and boundaryless Riemannian \emph{surface} which has \emph{constant} sectional curvature $K=\pm 1$. The unit tangent bundle of $M$ is defined by 
$$\mathcal{M}:=SM=\left\{q=(m,v)\in TM:\ \|v\|_{g(m)}=1\right\}.$$
There are two natural vector fields on $SM$: the geodesic vector field $X$ and the vertical vector field $V$, \emph{i.e.} the vector field corresponding to the action by rotation in the fibers of $SM$. One can then define $X_\perp:=[X,V]$ and these vector fields verify the following commutation relations~\cite[\S3.5.1]{PaternainSaloUhlmann22}:
\begin{equation}\label{e:commutators}
 [X_\perp,X]= \pm V,\quad [X,V]=X_\perp,\quad\text{and}\quad [X_\perp,V]=-X,
\end{equation}
where $\pm$ corresponds to the fact that the curvature is either $-1$ (minus case) or $1$ (plus case). The manifold $\mathcal{M}$ is naturally endowed with a Riemannian metric $g_S$ (the Sasaki metric) which makes $(X,X_\perp,V)$ into an orthonormal basis. The corresponding volume form that we will denote by $d\mu_L$ makes these three vector fields divergence free. Given a function $W\in\ml{C}^\infty(\ml{M},\IR)$, the goal of the present work is to study the asymptotic regularity of the eigenmodes of the (formally selfadjoint) operator
$$\ml{H}_h:=-h^2\Delta_{\text{sR}}+\delta_h^2W,\quad 0<h\leq 1$$
where $\Delta_{\text{sR}}:=X^2+X_\perp^2$ and where
\begin{equation}\label{e:strength-perturbation}
 \lim_{h\to 0^+}\delta_h=0,\quad\text{and}\quad\lim_{h\to 0^+}h\delta_h^{-1}=0.
 \end{equation}
 
 \begin{remark} The case $\delta_h=1$ was already considered in~\cite{ArnaizRiviere23} in the case of (variable) nonvanishing curvature (see Appendix~\ref{s:variable} for a brief reminder) while the restriction $\delta_h\gg h$ is maybe a shortcoming of our analysis (see \S\ref{r:term-order-3} for a more precise discussion).
 \end{remark}
 
Thanks to our curvature assumption, these operators are known to be hypoelliptic~\cite{Hormander67} and one can consider the Friedrichs extension of $\ml{H}_h$ which is the only semi-bounded selfadjoint extension\footnote{The domain $\mathcal{D}(\ml{H}_h)$ of this extension is independent of $h$ and verifies $H^2(\ml{M})\subset \mathcal{D}(\ml{H}_h)\subset H^1(\ml{M})$.} of $\ml{H}_h:H^2(\ml{M})\rightarrow L^2(\ml{M})$. This selfadjoint extension has compact resolvent and one can find an orthonormal basis of $L^2(\mathcal{M})$ made of eigenfunctions of $\ml{H}_h$ with eigenvalues $\lambda_j^2(h)\to \infty$ (as $j\to\infty$). By hypoelliptic regularity, these eigenmodes are smooth. See~\cite[App.~A]{ArnaizRiviere23} for a brief reminder on these spectral properties\footnote{Strictly speaking, this reference deals with the case where we replace $X$ by $V$ but the proof can be adapted verbatim to deal with the present set-up.}.

Our goal in the present work is to study the asymptotic regularity of these eigenfunctions,
\begin{equation}\label{e:eigenvalue-equation}
 \left(-h^2\Delta_{\text{sR}}+\delta_h^2W\right)\psi_h=E_h\psi_h,\quad\|\psi_h\|_{L^2}=1,\quad h>0,\quad\lim_{h\to 0^+} E_h=E_0>0,
\end{equation}
through the probability measures
$$\nu_{\psi_h}:a\in\ml{C}^0(\ml{M})\mapsto \int_{\ml{M}}a|\psi_h|^2d\mu_L.$$
More precisely, we say that a probability measure $\nu$ is a \emph{quantum limit} (for this spectral problem) if there exists a sequence $(\psi_{h_n},h_n)_{n\geq 1}$ of solutions to~\eqref{e:eigenvalue-equation} with $h_n\to 0^+$ such that 
$$\forall a\in\ml{C}^0(\ml{M}),\quad\lim_{n\rightarrow +\infty}\int_{\ml{M}}a|\psi_{h_n}|^2d\mu_L=\int_{\ml{M}}ad\nu.$$
In~\cite{ArnaizRiviere23}, we showed with Arnaiz that any such quantum limit $\nu$ can be decomposed as
\begin{equation}\label{e:decomposition-measure}
\nu=\nu_{\text{comp}}+\overline{\nu}_\infty+\sum_{k=0}^\infty\left(\nu_{k,\infty}^++\nu_{k,\infty}^-\right),
\end{equation}
where each term in the sum is a finite nonnegative Radon measure and where the measures $\overline{\nu}_\infty$ and $(\nu_{k,\infty}^\pm)_{k\geq 0}$ are all invariant by the flow $\varphi_V^t$ generated by the vertical vector field $V$. See Section~\ref{s:proof} for a brief reminder of the proof adapted to the setting of the present article. Recall also that $\nu_{\text{comp}}$ corresponds to the projection on $\ml{M}$ of the semiclassical measure~\cite{Burq97} obtained (up to extraction) from the sequence of solutions to~\eqref{e:eigenvalue-equation}. See~\S\ref{ss:compact} for more details.
\begin{remark}
In fact, the result in~\cite{ArnaizRiviere23} deals with the case where we replace $X$ by $V$ in the definition of $\Delta_{\text{sR}}$ but, as we shall see below, the geodesic flow in the setting of the present article is replaced by the periodic flow $\varphi_V^t$ induced by $V$. We also refer to~\cite{ColindeVerdiereHillairetTrelat15} for earlier results showing the invariance by $\varphi_V^t$ of the measure $\nu_{\infty}:=\nu-\nu_{\text{comp}}$ and to~\cite{BoilVuNgoc21, FermanianFischer21, FermanianLetrouit21, BurqSun22, ArnaizSun22} for other related models. As a consequence of this invariance by $\varphi_V^t$, it may happen that there exist subsequences of Schr\"odinger eigenfunctions such that $(\nu-\nu_{\text{comp}})(S_qM)>0$. We refer for instance to~\cite[\S3]{ColindeVerdiereHillairetTrelat15} for concrete examples in the case of sub-Riemannian Laplacians on compact quotients of the Heisenberg group $\mathbf{H}^1$ when $W\equiv 0$. More precisely, going through their argument, one can verify that, for any periodic orbit of the Reeb flow, one can find subsequences of eigenfunctions such that $(\nu_{0,\infty}^++\nu_{0,\infty}^-)$ puts its full mass on the periodic orbit. In fact, their construction would work as well for any choice of $k\geq0$. Our goal here is to show that, despite periodicity of the Reeb flow in our model, we can discard such concentration phenomena (at least for generic choices of potential $W$).
\end{remark}

\subsection{Main results}

The precise definition of the measures appearing in the decomposition~\eqref{e:decomposition-measure} is recalled in Section~\ref{s:chart} below. Let us just give an informal explanation before stating our main results. Due to the hypoelliptic nature of our problem, it follows that solutions to our eigenvalue problem oscillate at scales lying between $h^{-1}$ and $h^{-2}$ contrary to elliptic settings where they would oscillate exactly at the scale $h^{-1}$. See~\cite[\S8]{ArnaizRiviere23} for a concrete illustration of this phenomenon that is formulated in its full generality by the Rothschild-Stein Theorem~\cite{RothschildStein76}. The decomposition~\eqref{e:decomposition-measure} exactly captures these different scales of oscillations: the measure $\nu_{\text{comp}}$ describes the part of the solution oscillating at frequency $\omega\asymp h^{-1}$ while $\overline{\nu}_{\infty}$ the one oscillating at frequencies $h^{-1}\ll\omega\ll h^{-2}$. Finally, the measures $(\nu_{k,\infty}^\pm)_{k\geq 0}$ reflects the oscillations at scales $\omega\asymp h^{-2}$ which were shown to enjoy quantized features in~\cite{ArnaizRiviere23}. See Section~\ref{s:chart} for a precise formulation.

Our main Theorem shows that, in the specific setting of the article, the measures $(\nu_{k,\infty}^\pm)_{k\geq 0}$ enjoy more invariance properties (and thus more regularity):
\begin{theorem}\label{t:maintheo} Let $\nu$ be a quantum limit for the spectral problem~\eqref{e:eigenvalue-equation}. Then, for every $k\geq 0$ and for every $a\in\ml{C}^1(\ml{M})$, one has
 $$\int_{\ml{M}}\left(X\left(\widehat{W}_0\right)X_\perp-X_\perp\left(\widehat{W}_0\right)X\right)(a)d\nu_{k,\infty}^\pm=0,$$
 where
 $$\widehat{W}_0:=\frac{1}{2\pi}\int_0^{2\pi}W\circ\varphi_V^tdt.$$
\end{theorem}
The function $\widehat{W}_0$ can be identified with a function on $(M,g)$. Hence, we can define its gradient $\nabla_g\widehat{W}_0$ and also $\nabla_g^\perp \widehat{W}_0$ which is the vector field directly orthogonal to $\nabla_gW_0$ (with the same norm). As a corollary of this Theorem, we have
\begin{coro}\label{c:maincoro} Let $\nu$ be a quantum limit for the spectral problem~\eqref{e:eigenvalue-equation}. Then, for every $k\geq 0$ and for every $a\in\ml{C}^1(M)$, one has
 $$\int_{M}\nabla_g^\perp \widehat{W}_0(a)d\left(\Pi_*\nu_{k,\infty}^\pm\right)=0,$$
where $\Pi:(q,p)\in\ml{M}:=SM\mapsto q\in M$ is the canonical projection. In particular,
$$\sum_{k\geq 0}\left(\nu_{k,\infty}^+(S_{q_0}M)+\nu_{k,\infty}^-(S_{q_0}M)\right)>0\quad\Longrightarrow \quad q_0\in\operatorname{Crit}\left(\widehat{W}_0\right).$$ 
\end{coro}

These two results can be thought as hypoelliptic analogues of the results that we obtained together with Maci\`a on quantum limits for elliptic Schr\"odinger operators on Zoll manifolds~\cite{MaciaRiviere16, MaciaRiviere19} using the Weinstein averaging method~\cite{Weinstein77}. In that setting, quantum limits are lifted to semiclassical measures that are invariant by the (periodic) geodesic flow on the unit tangent bundle of these Zoll manifolds. We proved in these references that semiclassical measures enjoy some extra invariance properties which allowed us to discard concentration on certain closed geodesics under generic assumptions on the Zoll metric or on the potential. See~\cite{MaciaRiviere18, ArnaizMacia22, Riviere22} for further developments in that direction. 

It would also be natural to consider the case of variable (nonvanishing) sectional curvature $K$. A major difference is that, in this setting, the involved Reeb flow is given by $KV+X(K)X_\perp-X_\perp(K)X$ which is not periodic in general -- see Appendix~\ref{s:variable}. In particular, we cannot expect to apply the averaging method that will be used to prove Theorem~\ref{t:maintheo}. Yet, modulo some extra work, we emphasize that our method should in principle apply to more general sub-Riemannian contact Laplacians in dimension $3$ whose corresponding Reeb vector fields is periodic, e.g. in the flat Heisenberg case discussed in~\cite[\S3]{ColindeVerdiereHillairetTrelat15}. Finally, our method does not allow to deal with the part of the measure $\overline{\nu}_\infty$ even if it would also be interesting to understand its regularity properties.


\subsection{Relation with magnetic Laplacians on $(M,g)$}
Sub-Riemannian contact Laplacians in dimension $3$ are naturally connected with magnetic Laplacians in dimension $2$~\cite[\S3.2]{ColindeVerdiereHillairetTrelat15}. Indeed, such operators are locally modelled on the standard hypoelliptic Laplacian
$$\Delta_H=\left(\partial_x+\frac{y}{2}\partial_z\right)^2+\left(\partial_y-\frac{x}{2}\partial_z\right)^2,$$
acting on the $3$-dimensional Heisenberg group~\cite{RothschildStein76}. Roughly speaking, when $\Delta_H$ is applied to test functions of the form $u(x,y)e^{iBz}$, we recognize the magnetic Laplacian on $\IR^2$ with constant magnetic field $B$. In our setting, the connection with magnetic Laplacians on the surface $M$ is \emph{global} and particularly explicit once one has observed that
there exists a natural decomposition of the space $L^2(\ml{M})$ by letting
\begin{equation}\label{e:fourier-decomposition}
\forall f\in L^2(\ml{M}),\quad f=\sum_{n\in\IZ}\widehat{f}_n,
 \end{equation}
where 
$$\widehat{f}_n(m,v):=\frac{1}{2\pi}\int_0^{2\pi} f\circ\varphi_V^t(m,v)e^{-int}dt.$$
This yields a decomposition $L^2(\ml{M}):=\oplus_{n\in\IZ}L^2_n(\ml{M})$, where each space $L^2_n(\ml{M})$ can be identified with the space of sections $L^2(M,\kappa^{\otimes n})$ for the canonical line bundle $\kappa$ induced by $g$ on $M$ and the anticanonical one $\kappa^{-1}\simeq\kappa^*$~\cite[\S6.1]{PaternainSaloUhlmann22}. \emph{If the potential $W$ is a function on the base} $M$ (i.e. independent of $v$), then the operator $\mathcal{H}_h$ can be restricted to these spaces
$$\mathcal{H}_{h,n}:=\mathcal{H}_h|_{\ml{C}^\infty_n(\ml{M})}=h^2\mathbf{A}^*_{n+1}\mathbf{A}_n-nh^2K+\delta_h^2W:\ml{C}^\infty_n(\ml{M})\rightarrow \ml{C}^\infty_n(\ml{M}),$$
where $\mathcal{C}^\infty_n(\ml{M})\simeq\Gamma^\infty(M,\kappa^{\otimes n})$ is the space of smooth sections of $\kappa^{\otimes n}$ and $$\mathbf{A}_n:=(X+iX_\perp):\mathcal{C}^\infty_n(\ml{M})\rightarrow \mathcal{C}^\infty_{n+1}(\ml{M}).$$ Recall that, for every $n\in\IZ$, one has $\mathbf{A}_{n+1}^*=(-X+iX_\perp):\mathcal{C}^\infty_{n+1}(\ml{M})\rightarrow \mathcal{C}^\infty_n(\ml{M})$. Hence, when considering sequences of solutions to~\eqref{e:eigenvalue-equation} lying in the space $\ml{C}^\infty_n(\ml{M})$, one recovers eigenfunctions of the magnetic (or horizontal) Laplacian $\ml{H}_{h,n}:\Gamma^\infty(M,\kappa^{\otimes n})\rightarrow\Gamma^\infty(M,\kappa^{\otimes n})$ and the limit $|n|\to\infty$ corresponds to the limit of strong magnetic fields of strength $|n|$ that are proportional to the volume form $\Omega:=K\text{Vol}_g$ on $(M,g)$. Equivalently, studying magnetic eigenfunctions amounts to considering joint eigenfunctions of $V$ and $\ml{H}_h$.

When $K\equiv\pm 1$ and $W\equiv 0$, the ``low eigenvalues'' of $\ml{H}_{1,n}$ (as $|n|\to\infty$) are of the form $(2k+1)n+\ml{O}_k(1)$ and have high multiplicity $\asymp_k|n|$, the so-called Landau levels~\cite{Demailly85, GuilleminUribe88} (or~\cite[Th.10.2.2]{FaureTsujii15} for a statement close to ours). See also~\cite{HelfferKordyukov11, RaymondVuNgoc15, Morin19, Charles20, Charles21, KordyukovTaimanov22, Kordyukov22b} for developments on these Landau levels in various geometric settings and~\cite{Morin22} for a recent review regarding magnetic Laplacians on Riemannian manifolds. A natural setting is thus to consider these joint eigenfunctions with $h_n^2|n|=1$ so that the low eigenvalues of $\mathcal{H}_{h_n,n}$ are of size $\asymp1$.
Note also that the perturbation by the potential is of order $1/n\ll\delta_{h_n}^2\ll 1$, and thus asymptotically small (see Appendix~\ref{s:variable} for a discussion on the case $\delta_h=1$). In particular, it results into low eigenvalues clusters that do not overlap in the semiclassical limit. We refer to~\cite{GuilleminUribe86} for results on the distributions of eigenvalues in these clusters. For these low eigenvalues, the corresponding measures $(\nu_{k,\infty}^\pm)_{k\geq 0}$ carry the full mass of $\nu$ as the frequency of oscillation $|n|$ is precisely of order $h_n^{-2}$. Hence, when specified to joint eigenfunctions of $V$ and $\mathcal{H}_h$, our results yield informations on the regularity of low-energy eigenmodes of the semiclassical magnetic Laplacians,
$$
\mathcal{H}_{h_n,n}:=\frac{1}{n}\mathbf{A}_{n+1}^*\mathbf{A}_n+\delta_{h_n}^2W,\quad \frac{1}{n}\ll\delta_{h_n}^2\ll 1,
$$
in the limit $n\to\infty$.

Similarly, with our conventions (and for $W\equiv 0$), the measure $\overline{\nu}_{\infty}$ describes magnetic eigenmodes of $\mathcal{H}_{1,n}$ for eigenvalues $|n|\ll \lambda=h^{-2}\ll |n|^2$ while $\nu_{\text{compact}}$ capture the behaviour of eigenmodes with eigenvalues $\lambda\gtrsim |n|^2$. Hence, on the one hand, the measures $(\nu_{k,\infty}^\pm)_{k\geq 0}$ \emph{do not a priori describe these high energy eigenmodes} of $\ml{H}_{1,n}$ as it is the case in~\cite{GuilleminUribe89, SchraderTaylor89, Zelditch92b}. For instance, in~\cite{Zelditch92b}, Zelditch showed that, when $K=-1$, orthonormal families of eigenfunctions associated with eigenvalues $\lambda \asymp\beta n^2$ (for $\beta$ larger than some critical value $\beta_c>0$) verify a quantum ergodicity property as $|n|\to\infty$. See also~\cite{GuilleminUribe86, GuilleminUribe88, GuilleminUribe89, SchraderTaylor89} for earlier related works in various geometric set-ups. On the other hand, coming back to our hypoelliptic problem~\eqref{e:eigenvalue-equation}, our main Theorem provides slightly more precise informations in the sense that it deals with eigenfunctions of the full operator $\ml{H}_h=\oplus_{n\in\IZ}\ml{H}_{h,n}$ where different semiclassical levels $n\in\IZ$ may be involved and where $W$ may be a function depending on the angular variable\footnote{In that case, $\mathcal{H}_h$ does not induce an operator $\mathcal{H}_{h,n}$ on $\ml{C}^\infty_n(\ml{M})$.} $v\in S_mM$.

\subsection{Organization of the article} We follow the strategy from~\cite{ArnaizRiviere23}. To that aim, we introduce some conventions in Section~\ref{s:chart} and we briefly review the results from this reference that allow to microlocalize the solutions to~\eqref{e:eigenvalue-equation} in the region of phase space where $|V|\asymp h^{-2}$. Then, in Section~\ref{s:normalform}, we slightly refine the normal form procedure from~\cite{ArnaizRiviere23} in view of analyzing the lower order terms in the semiclassical expansion arising from the eigenvalue equation~\eqref{e:eigenvalue-equation}. This refined normal form is then implemented in Section~\ref{s:proof} where the main technical result of the article is proved. Finally, in Section~\ref{s:endproof}, we show how it allows to end the proof of Theorem~\ref{t:maintheo}.

The article also contains three appendices. In Appendix~\ref{s:variable}, we briefly discuss what differs in the variable curvature case and what can already be said using the results from~\cite{ArnaizRiviere23}. Then, in Appendix~\ref{a:compact}, we discuss the properties of the measure $\nu_{\text{comp}}$ (more precisely of its microlocal lift). In particular, we compare more precisely our framework with the one from~\cite{Zelditch92b} on high energy joint eigenfunctions of the Kaluza-Klein Laplacian $X^2+X_\perp^2+V^2$ and $V$ (when $K\equiv -1$). Finally, in Appendix~\ref{a:pdo}, we collect a few standard results from semiclassical analysis that are used all along the article.

\subsection*{Acknowledgements} I address my deepest thanks to V\'ictor Arnaiz for countless discussions on hypoelliptic eigenmodes in Nantes during the years 2022-2023. I also warmly thank Laurent Charles, Yannick Guedes Bonthonneau, Thibault Lefeuvre, L\'eo Morin and San Vu Ngoc for their insights on magnetic Laplacians in dimension $2$.

\section{Semiclassical preliminaries}
 \label{s:chart}
As in~\cite{ArnaizRiviere23}, we will work in a system of isothermal coordinates near a point $m_0$ in $(M,g)$. Namely, we fix a system of local coordinates $(x,y)\in U_0$ near the origin in $\IR^2$ (with $(0,0)$ being the image of $m_0$) such that the metric $g$ writes down in a conformal way $g=e^{2\lambda(x,y)}(dx^2+dy^2).$
\begin{remark}\label{r:compex-structure} If we fix an atlas $(U_j,\psi_j)_{j\in J}$ made of isothermal charts, then one can verify that the induced maps $\psi_{i}\circ\psi_j^{-1}:\IC\rightarrow\IC$ are holomorphic and that they endow $(M,g)$ with a complex structure which is independent of the coordinate charts~\cite[Ch.3]{PaternainSaloUhlmann22}. Recall also that the complex structure depends only on the conformal class of $g$. The (resp. anti) canonical line bundle $\kappa$ (resp. $\kappa^{-1}$) mentionned in the introduction is the line bundle of $(1,0)$ (resp. $(0,1)$)-forms, i.e. proportional to $dx+idy$ (resp. $dx-idy$) in local isothermal coordinates. 
\end{remark}

Without loss of generality we can extend $\lambda$ into a smooth compactly supported function on $\IR^2$. This system of coordinates on $M$ naturally induces a system of coordinates near $S_{m_0}M$ by letting $z$ be the angle between a unit vector $p\in S_qU_0$ and $\frac{\partial}{\partial x}$. With these coordinates at hand, the measure $\mu_L$ writes down
$$d\mu_L(x,y,z)=e^{2\lambda(x,y)}dxdydz,$$
while the vector fields of interest are given by
\begin{equation}\label{e:expressionX}
 X:=e^{-\lambda}\left(\cos z\partial_x+\sin z\partial_y+\left(-\sin z\partial_x\lambda +\cos z\partial_y\lambda\right)\partial_z\right),
\end{equation}
\begin{equation}\label{e:expressionXperp}
 X_\perp:=e^{-\lambda}\left(\sin z\partial_x-\cos z\partial_y+\left(\cos z\partial_x\lambda +\sin z\partial_y\lambda\right)\partial_z\right),
\end{equation}
and 
\begin{equation}\label{e:expressionV}
 V:=\partial_z.
\end{equation}
See~\cite[Ch.~3]{PaternainSaloUhlmann22} for details.
\begin{remark}
As $\lambda$ (and $W$) has been extended into a smooth compactly supported function on $\IR^2\times\IS^1$, we have globally well defined operators $X$, $X_\perp$, $V$ and $\mathcal{H}_h$ on $\IR^2\times \IS^1$. We keep the same notation for simplicity even if they only coincide inside the chart $\mathcal{U}_0:=U_0\times\IS^1$ with our operators.
\end{remark}

\subsection{Semiclassical formulation}

It will be more convenient to work with the standard Lebesgue measure in our local chart. Hence, we define the conjugated operator
$$\widehat{P}_h:=e^{\lambda}\left(-h^2\Delta_{\text{sR}}+\delta_h^2W\right)e^{-\lambda},$$
which has a somewhat simpler expression to deal with. Indeed, one has
$$e^{\lambda}\frac{h}{i}Xe^{-\lambda}=\frac{h}{i}X+ihX(\lambda)=\Op_h^w(H_1),$$
where $\Op_h^w$ is the Weyl quantization on $T^*(\IR^2\times\IS^1)$ (see Appendix~\ref{a:pdo} for a brief reminder) and where $H_1(q,p):=p(X(q))$ is the principal symbol of the operator $-ihX$. Similarly, one has 
$e^{\lambda}\frac{h}{i}X_\perp e^{-\lambda}=\Op_h^w(H_2),$
where $H_2(q,p):=p(X_\perp(q))$ is the principal symbol of the operator $-ihX_\perp$. For the following, we also set $H_3(q,p)=p(V(q))$ to be the principal symbol of $\frac{h}{i}V$. See~\cite[\S3]{ArnaizRiviere23} for more details on these conventions.

Gathering these two observations, one can write 
\begin{equation}\label{e:expression-operator}
 \widehat{P}_h=\Op_h^w\left(H_1^2+H_2^2+\delta_h^2W+h^2W_\lambda\right),
\end{equation}
where $W_\lambda$ is the remainder coming from the terms of order $2$ in the Weyl quantization of our symbols (and depending on our choice of coordinate charts). As we aim at understanding the influence of the lower order terms, we need to compute the remainder term $W_\lambda$ somewhat explicitly. This is the content of the following Lemma:
\begin{lemma}\label{l:expression-remainder} One has
$$W_\lambda=\frac{1}{2}e^{-2\lambda}\left(\partial_x^2\lambda+\partial_y^2\lambda\right).$$
In particular, if the sectional curvature $K$ is equal to $\pm 1$, then, near $0$, $W_\lambda=\mp\frac{1}{2}.$
\end{lemma}
\begin{proof} The proof is just an explicit calculation using the composition formula~\eqref{e:composition-formula} for the Weyl quantization. Namely, as $H_1(q,p)$ and $H_2(q,p)$ are polynomials of degree $1$, this extra term is given by
$$\frac{1}{4}\sum_{|\alpha|=|\beta|=1}\left(\partial_p^\beta\partial_q^\alpha H_1\partial_p^\beta\partial_q^\alpha H_1+\partial_p^\beta\partial_q^\alpha H_2\partial_p^\beta\partial_q^\alpha H_2\right).$$
Then, a (somehow tedious) direct calculation based on the exact expressions~\eqref{e:expressionX} and~\eqref{e:expressionXperp} gives the expected result. Note that the expression for $W_\lambda$ is independent of the fact that the curvature is constant. 
\end{proof}

Coming back to our problem and given a solution $\psi_h$ to the eigenvalue problem~\eqref{e:eigenvalue-equation}, we set 
\begin{equation}\label{e:conjugated-eigenfunction}
 u_h:=e^{\lambda}\psi_h
\end{equation}
so that it solves locally in $\mathcal{U}_0$ the eigenvalue equation
\begin{equation}\label{e:semiclassical-eigenvalue}
 \widehat{P}_h u_h=E_hu_h.
\end{equation}
Thanks to the Rothschild-Stein Theorem~\cite{RothschildStein76}, one can verify that, for any compact subset $\mathcal{K}$ of $\mathcal{U}_0$,
\begin{equation}\label{e:apriori-semiclassical}
\left\|\Op_h^w(H_1)u_h\right\|_{L^2(\ml{K})}+\left\|\Op_h^w(H_2)u_h\right\|_{L^2(\ml{K})}+\left\|\Op_h^w(hH_3)u_h\right\|_{L^2(\ml{K})}\leq C_{\mathcal{K}},
\end{equation}
where $C_\mathcal{K}>0$ is independent of $(\psi_h,h)$ solving the eigenvalue problem and where the $L^2$ norm is taken with respect to the standard Lebesgue measure on $\IR^2\times\IS^1$. See~\cite[Lemma~A.5]{ArnaizRiviere23} for more details. 

\subsection{Microlocalization at infinity}

We now recall how one can obtain the decomposition~\eqref{e:decomposition-measure} of a quantum limit $\nu$ by working in the above local coordinate chart (it is sufficient by a partition of unity argument).

We fix for the rest of the article a smooth function $\chi:\IR\rightarrow [0,1]$ which is equal to $1$ on $[-1,1]$ and to $0$ outside $[-2,2]$. Moreover, we make the assumption that $\chi'\geq 0$ on $\IR_-$ and $\chi'\leq 0$ on $\IR_+$. For such a function, we also set
$\tilde{\chi}=1-\chi.$ 

\subsubsection{Reduction to the region $1\ll |H_3|\lesssim h^{-1}$}

First, we recall from~\cite[Lemma 3.3]{ArnaizRiviere23} that, for every $0<\varepsilon<1$ and for every $R>1$, the functions
\begin{equation}\label{e:cutoffball}
 \chi_R^B:=\chi\left(\frac{H_1^2+H_2^2+H_3^2}{R}\right),\quad\tilde{\chi}_R^B:=1-\chi_R^B,
\end{equation}
and
\begin{equation}\label{e:cutoffcone}
 \chi_\varepsilon^C:=\chi\left(\frac{\varepsilon H_3}{\sqrt{1+H_1^2+H_2^2}}\right),\quad\tilde{\chi}_\varepsilon^C:=1-\chi_\varepsilon^C
\end{equation}
belong\footnote{Strictly speaking, the proof in this reference deals with the case where the role of $H_1$ and $H_3$ are intertwined but this does not affect the argument there.} to the class of symbols $S^0_{\text{cl}}(T^*(\IR^2\times\IS^1))$ amenable to pseudodifferential calculus (whose definition is recalled in Appendix~\ref{a:pdo}). Similarly, thanks to the proof of~\cite[Lemma~3.8]{ArnaizRiviere23} and letting $b\in\ml{C}^{\infty}_c(\ml{U}_0\times\IR)$, the function
$b(x,y,z,hH_3)\tilde{\chi}_\varepsilon^C\tilde{\chi}_R^B$ belongs to that same class of symbols (with seminorms that are uniformly bounded in terms of $0<h\leq h_0$).

The main object of study in~\cite{ArnaizRiviere23} was the sequences of distributions:
\begin{equation}\label{e:wigner}
 \mu_h^{R,\varepsilon}:b\in\ml{C}^{\infty}_c(\ml{U}_0\times\IR)\mapsto \langle\Op_h^w\left(b(x,y,z,hH_3)\tilde{\chi}_\varepsilon^C\tilde{\chi}_R^B\right)u_h,u_h\rangle_{L^2},
\end{equation}
where the scalar product is taken with respect to the Lebesgue measure in the chart and where $u_h$ is locally defined by~\eqref{e:conjugated-eigenfunction}. These distributions exactly capture the part of $u_h$ oscillating at frequencies $h^{-1}\ll\omega\lesssim h^{-2}$.

For latter purpose, we also record the following useful upper bound that was obtained, for every $k\geq 0$, in the proof of Lemma 4.9 from this reference:
\begin{multline}\label{e:aprioribound-subtle}\left\|\Op_h^w\left(b(x,y,z,hH_3)\tilde{\chi}_\varepsilon^C\tilde{\chi}_R^B\right)A_h^ku_h\right\|_{L^2}\\+\left\|\Op_h^w\left(b(x,y,z,hH_3)\tilde{\chi}_\varepsilon^C\tilde{\chi}_R^B\right)(A_h^*)^ku_h\right\|_{L^2}\leq C_{R,\varepsilon, k,b},
\end{multline}
where $C_{R,\varepsilon, k,b}>0$ is some positive constant that is independent of $h$ (it only depends on the various parameters appearing in the index) and where 
\begin{equation}\label{e:def-Ah}
 A_h:=\Op_h^w(H_1+iH_2).
\end{equation}

\subsubsection{Reminder on the support properties of the limit measure}\label{ss:limit-measure}

Recall now from~\cite[\S 3.4]{ArnaizRiviere23} that, up to successive extractions, we can suppose that there exists a finite (nonnegative) Radon measure $\mu_\infty$ on $\ml{U}_0\times\IR$ such that, for every $b\in\ml{C}^{\infty}_c(\ml{U}_0\times\IR)$,
$$\lim_{\varepsilon\to 0}\lim_{R\to\infty}\lim_{h\to 0}\langle\mu_h^{R,\varepsilon},b\rangle=\int_{\ml{U}_0\times\IR}bd\mu_\infty.$$
Moreover, according to Proposition~4.1 in that same reference (that can be adapted verbatim to our case), the measure $\mu_\infty$ can be decomposed as
\begin{equation}\label{e:decomposition-mu-infty}
\mu_\infty=\overline{\mu}_\infty+\sum_{k=0}^{\infty}\left(\mu_{k,\infty}^++\mu_{k,\infty}^-\right),
\end{equation}
where $\overline{\mu}_\infty$ and $(\mu_{k,\infty}^\pm)_{k\geq 0}$ are nonnegative finite Radon measure, where $\overline{\mu}_\infty$ is supported inside $\ml{U}_0\times\{0\}$ and where $\mu_{k,\infty}^\pm$ is supported in the set
\begin{equation}\label{e:support-measure}
\ml{V}_k:=\left\{(q,E)\in\ml{U}_0\times\IR^*:E=\pm\frac{E_0}{2k+1}\right\}.
 \end{equation}
Finally, recall from~\cite[\S7]{ArnaizRiviere23} that the relation with the measures appearing in~\eqref{e:decomposition-measure} is as follows:
\begin{equation}\label{e:relation-mu-nu-0}
\overline{\mu}_\infty(q,E)=\overline{\nu}_\infty(q)\otimes\delta_0(E),
 \end{equation}
and, for every $k\geq 0$,
\begin{equation}\label{e:relation-mu-nu-k}
\mu_{k,\infty}^\pm(q,E)=\nu_{k,\infty}^\pm(q)\otimes\delta_0\left(E\mp\frac{E_0}{2k+1}\right).
\end{equation}

\subsection{The measure $\nu_{\text{comp}}$}\label{ss:compact} According to~\cite[\S 3.1]{ArnaizRiviere23}, the measure $\nu_{\text{comp}}$ is obtained as the limit as $h\to 0^+$ and $R\to \infty$ (in this order) of the distribution
$$\nu_{h,R}:a\in\ml{C}^\infty_c(\ml{U}_0)\mapsto\left\langle\Op_h^w(a\chi_R^B)u_h,u_h\right\rangle_{L^2}\in\IC.$$
Up to another extraction, we can suppose that the following sequence of distributions,
$$w_h:a\in\ml{C}^\infty_c(T^*\ml{U}_0)\mapsto\left\langle\Op_h^w(a)u_h,u_h\right\rangle_{L^2}\in\IC,$$
converges as $h\to 0^+$ to some limit distribution $w\in\ml{D}'(T^*\ml{U}_0)$ (and by partition of unity on $T^*\ml{M}$). Following~\cite[Ch.~5]{Zworski12}, $w$ is a finite nonnegative measure carried by the \emph{noncompact} set $\{H_1^2+H_2^2=E_0\}$ and invariant by the Hamiltonian flow of $H_1^2+H_2^2$. Moreover, $\nu_{\text{comp}}=\tilde{\Pi}_*(w)$, where $\tilde{\Pi}:T^*\ml{M}\rightarrow \ml{M}$ is the canonical projection.  We refer to Appendix~\ref{a:compact} for more details on the nature of the (magnetic type) flow generated by this subelliptic Hamiltonian.

When $W$ is the pullback of a function on $M$ (i.e. independent of the $z$-variable), it is also interesting to consider joint eigenfunctions of $\mathcal{H}_h$ and $hV$ in view of the relation with magnetic Laplacians. More precisely, one can consider solutions to~\eqref{e:eigenvalue-equation} that also verify
\begin{equation}\label{e:magnetic-eigenvalues}
 hV\psi_h=B_h\psi_h,\quad B_h\rightarrow B\in\IR\cup\{\pm\infty\}\quad\text{as}\quad h\rightarrow 0^+. 
\end{equation}
Recall that, from the Rothschild-Stein Theorem, $B_h=\ml{O}(h^{-1})$. When $B\in\IR$, one can verify that $w$ is a probability measure, i.e. $\nu_{\text{comp}}(\ml{M})=1$, and that it is in addition invariant under the Hamiltonian flow generated by $H_3$ and carried by the compact set $\{H_1^2+H_2^2=E_0\}\cap \{H_3=B\}$. On the opposite case where $B=\pm \infty$, one has $\nu_{\text{comp}}(\ml{M})=0$ and, depending on the rate of convergence of $B_h$ to infinity, the measure $\overline{\nu}_\infty$ carries the full mass (when $B_h\ll h^{-1}$) or one of the measure $\nu_{k,\infty}^\pm$ carries it (when $B_h\asymp h^{-1}$). 

\section{Normal form procedure}\label{s:normalform}

The proof of the invariance properties of $\mu_\infty$ in~\cite{ArnaizRiviere23} relied on a normal form procedure adapted to the geometry of the problem and inspired from~\cite{ColindeVerdiereHillairetTrelat15}. The fact that the operators are slightly different than in~\cite{ArnaizRiviere23} will modify the terms appearing in the normal form. We do not discuss all the details and we just focus on the main differences, refering to~\cite[\S5]{ArnaizRiviere23} for more explanations on this procedure and references. The main point compared with~\cite{ArnaizRiviere23} is that we need to keep track of more terms in view of analyzing the influence of the potential $W$.

The strategy is to replace a test function $a(x,y,z)$ by a function $\mathbf{a}(x,y,z,\xi,\eta,\zeta)$ whose Poisson bracket with $H_1^2+H_2^2$ is as small as possible in the regime $|H_1|+|H_2|\ll|H_3|$. To do this, we set $Z=H_1+iKH_2$, where we recall that the sectional curvature $K$ is constant equal to $\pm 1$. Thanks to~\eqref{e:commutators}, one has
\begin{equation}\label{e:commutatorZ}
\{Z,\overline{Z}\}=2i H_3,
\end{equation}
from which we deduce the key observation in view of performing our normal form procedure:
\begin{equation}\label{e:cohomological}
\left\{|Z|^2,\frac{Z^k\overline{Z}^l}{2i(l-k)}\right\}=H_3Z^k\overline{Z}^l.
 \end{equation}

\begin{remark}\label{r:commutatorH3} Note that a key simplification compared with~\cite{ArnaizRiviere23} is that, thanks to~\eqref{e:commutators},
$$\{H_1^2+H_2^3,H_3\}=2H_1\{H_1,H_3\}+2H_2\{H_2,H_3\}=0.$$
In particular, $H_3$ is already in normal form with respect to $H_1^2+H_2^2$.
\end{remark}
 The first step of the normal form procedure consists in observing that
 $$\{|Z|^2,a\}=Z\{\overline{Z},a\}+\overline{Z}\{Z,a\},$$
 and, in view of~\eqref{e:commutatorZ}, in letting 
 $$\tilde{a}_1=\frac{Z}{2iH_3}\{\overline{Z},a\}-\frac{\overline{Z}}{2iH_3}\{Z,a\}.$$
 Hence, if we set $a_1=a+\tilde{a}_1$, we get
 $$\{|Z|^2,a_1\}=\frac{|Z|^2}{H_3}V(a)+\frac{Z^2}{2iH_3}X_{\overline{Z}}^2(a)-\frac{\overline{Z}^2}{2iH_3}X_Z^2(a),$$
 where $X_Z=X+iKX_\perp.$ We would now like to eliminate the terms involving $Z^2$ and $\overline{Z}^2$ so that we define 
 $$\tilde{a}_2:=-\frac{Z^2}{8H_3^2}X_{\overline{Z}}^2(a)-\frac{\overline{Z}^2}{8H_3^2}X_Z^2(a),$$
and $a_2:=a_1+\tilde{a}_2$. It yields the following simplification:
$$
\{|Z|^2,a_2\}=\frac{|Z|^2}{H_3}V(a)-\frac{Z^3}{8H_3^2}X_{\overline{Z}}^3(a)-\frac{\overline{Z}^3}{8H_3^2}X_Z^3(a)
-\frac{Z^2\overline{Z}}{8H_3^2}X_ZX_{\overline{Z}}^2(a)-\frac{\overline{Z}^2Z}{8H_3^2}X_{\overline{Z}}X_Z^2(a).
$$
We can iterate this procedure and pick $\tilde{a}_3$ and $\tilde{a}_4$ that are of the form
$$\tilde{a}_j=\frac{1}{H_3^j}\sum_{|\alpha+\beta|=j}Z^\alpha\overline{Z}^\beta \mathcal{L}_{\alpha,\beta}(a),\quad j=3,4,$$
with $\mathcal{L}_{\alpha,\beta}$ being differential operators of order $\leq j$, and such that, if we set $\mathbf{a}=a+\sum_{j=1}^4\tilde{a}_j$, one obtains
\begin{equation}\label{e:commutator-a}
\{|Z|^2,\mathbf{a}\}=\frac{|Z|^2}{H_3}V(a)+\frac{|Z|^4}{16iH_3^3}\left(X_{\overline{Z}}^2X_Z^2-X_Z^2X_{\overline{Z}}^2\right)(a)+\frac{1}{H_3^4}\sum_{|\alpha+\beta|=5}Z^{\alpha}\overline{Z}^\beta \tilde{\mathcal{L}}_{\alpha,\beta}(a),
\end{equation}
with $\tilde{\mathcal{L}}_{\alpha,\beta}$ being differential operators of order $\leq 5$. Note that the involved differential operators depend only on the coordinate chart (but not on $a$). Using the commutation properties~\eqref{e:commutators} of our operators, this can be rewritten as
\begin{equation}\label{e:commutator-a-real}
\{|Z|^2,\mathbf{a}\}=\frac{H_1^2+H_2^2}{H_3}V(a)-\frac{(H_1^2+H_2^2)^2}{2H_3^3}\Delta_{\text{sR}}V(a)+\frac{1}{H_3^4}\sum_{|\alpha+\beta|=5}H_1^{\alpha}H_2^\beta \tilde{\mathcal{R}}_{\alpha,\beta}(a),
\end{equation}
where $\tilde{\mathcal{R}}_{\alpha,\beta}$ are differential operators of order $\leq 5$. Finally, we record the expression of $\mathbf{a}$:
\begin{equation}\label{e:normal-form-a}
 \mathbf{a}=a+\frac{H_2}{H_3}X(a)-\frac{H_1}{H_3}X_\perp(a)+\sum_{j=2}^4\frac{1}{H_3^j}\sum_{|\alpha+\beta|=j}H_1^{\alpha}H_2^\beta \mathcal{R}_{\alpha,\beta}(a),
\end{equation}
where $\mathcal{R}_{\alpha,\beta}$ are differential operators of order $\leq |\alpha+\beta|$.
\begin{remark} Thanks to~\cite[Cor.~2.6]{ArnaizRiviere23}, the symbol $\mathbf{a}$ belongs to the class of symbols $S^0_{\text{cl}}(T^*(\IR^2\times\IS^1))$ amenable to pseudodifferential calculus inside the support of $\tilde{\chi}_{\varepsilon}^C\tilde{\chi}_R^B.$ 
\end{remark}

\section{Invariance of the measure at infinity}\label{s:proof}
The key property for our analysis is the following extra-invariance property of our limit measures:
\begin{prop}\label{p:mainprop} 
With the above conventions, for every $a\in\ml{C}^1_c(\ml{U}_0)$ verifying $V(a)=0$ and for every $k\in\IZ_+$, one has
 $$\nu_{k,\infty}^\pm (X(W)X_\perp(a)-X_\perp(W)X(a))=0.$$
\end{prop}

\subsection{Proof of Proposition~\ref{p:mainprop}}
The proof follows the lines of~\cite[\S6]{ArnaizRiviere23} and it was itself inspired by the microlocal proof initiated in~\cite{ColindeVerdiereHillairetTrelat15}. See also~\cite{BurqSun22, ArnaizSun22} for related arguments in the case of Baouendi-Grushin operators. The main input compared with~\cite{ArnaizRiviere23} is that we analyze the lower order terms in the asymptotic expansion.

We let $a$ be a smooth \emph{real-valued} function that is compactly supported inside $\ml{U}_0$ and, for every 
$k\in\IZ_+$, we let $\theta_k(E)$ be an element in $\ml{C}^\infty_c(\IR)$ which is compactly supported inside the interval $\left(\frac{E_0}{2k+2},\frac{E_0}{2k}\right)$ and which is equal to $1$ in a small neighborhood of $\frac{E_0}{2k+1}$. We set $\theta_{\pm k}\left(E\right):=\theta_k(\pm E).$ Regarding~\eqref{e:relation-mu-nu-k} and in view of analyzing $\nu_{k,\infty}^\pm$, the strategy from~\cite{ArnaizRiviere23} consists in picking the test function
$$b_h(x,y,z,\xi,\eta,\zeta):=H_3\theta_{\pm k}(hH_3)\left(\tilde{\chi}_R^B\tilde{\chi}_{\varepsilon}^C\mathbf{a}\right)(x,y,z,\xi,\eta,\zeta),$$
where $\mathbf{a}$ is constructed from $a$ using~\eqref{e:normal-form-a}. We note, from the definitions~\eqref{e:cutoffball} and~\eqref{e:cutoffcone} of our cutoff functions and from Remark~\ref{r:commutatorH3}, that 
\begin{equation}\label{e:trivial-commutator}
\{H_1^2+H_2^2,\tilde{\chi}_R^B\tilde{\chi}_{\varepsilon}^C\theta_{\pm k}(hH_3)H_3\}=0.
\end{equation}
 Now using the eigenvalue equation~\eqref{e:semiclassical-eigenvalue} under the form 
 \begin{equation}\label{e:bracket-eigenvalue}
  \left\langle\left[\Op_h^w(H_1^2+H_2^2+\delta_h^2W),\Op_h^w\left(\mathbf{a}\tilde{\chi}_R^B\tilde{\chi}_{\varepsilon}^C\theta_{\pm k}(hH_3)H_3\right)\right]u_h,u_h\right\rangle=\ml{O}(h^\infty),
 \end{equation}
together with the composition rule~\eqref{e:composition-formula} for the Weyl quantization applied up to $\ml{O}(h^3)$ remainders, we get
 $$
 \left\langle\Op_h^w\left(\{H_1^2+H_2^2,\mathbf{a}\}\tilde{\chi}_R^B\tilde{\chi}_{\varepsilon}^C\theta_{\pm k}(hH_3)H_3\right)u_h,u_h\right\rangle=\ml{O}_{R,\varepsilon}(h^2)+\ml{O}_{R,\varepsilon}(\delta_h^2).
 $$
Here we used the fact that the symbol involved in our test function lies in the class $S^1_{\text{cl}}(T^*(\IR^2\times\IS^1))$. Once this is written, the proof in~\cite{ArnaizRiviere23} is achieved by using the normal form equation~\eqref{e:commutator-a-real} together with the a priori estimates~\eqref{e:apriori-semiclassical}. This yields
$$\left\langle\Op_h^w\left(\left(V(a)(H_1^2+H_2^2)\right)\tilde{\chi}_R^B\tilde{\chi}_{\varepsilon}^C\theta_{\pm k}(hH_3)\right)u_h,u_h\right\rangle=\ml{O}(\varepsilon)+\ml{O}_{R,\varepsilon}(\delta_h^2)+\ml{O}_{R,\varepsilon}(h^2).$$
Recalling~\eqref{e:semiclassical-eigenvalue} and~\eqref{e:apriori-semiclassical} and letting $h\to 0^+$, $R\to \infty$ and $\varepsilon\to 0$ in this order, one finds, using~\eqref{e:decomposition-mu-infty} and~\eqref{e:relation-mu-nu-k}, that $\nu_{k,\infty}^\pm(V(a))=0$ (as $E_0>0$). 

We now want to show our extra invariance properties. To that aim, we pick a test function $a$ in $\ml{C}^\infty_c(\ml{U}_0)$ verifying $V(a)=0$ (equivalently that does not depend on the $z$-variable) and we revisit the above argument more carefully.

\begin{remark} The choice of the local chart is quite important here as it allows to pick symbols $a$ that are invariant by the flow $\varphi_V^t$ and that are supported inside the chart. By a partition of unity argument, any $V$-invariant test function can be treated through this local procedure thanks to our choice of coordinate chart. 
\end{remark}

Thanks to~\eqref{e:commutator-a-real} and to~\eqref{e:normal-form-a}, it follows from~\eqref{e:bracket-eigenvalue} together with the composition rule~\eqref{e:composition-formula} for $\Op_h^w$ applied up to $\ml{O}(h^5)$ remainders that, up to $\ml{O}(h^2)$ remainders, we are left with analyzing the four following terms:
\begin{enumerate}
 \item the remainder coming from the normal form:
 \begin{equation}\label{e:remainder-normal-form}
 h\sum_{|\alpha+\beta|=5}\left\langle\Op_h^w\left(\tilde{\chi}_R^B\tilde{\chi}_{\varepsilon}^C\theta_{\pm k}(hH_3)\frac{Z^{\alpha}\overline{Z}^\beta}{(hH_3)^3} \tilde{\mathcal{R}}_{\alpha,\beta}(a)\right)u_h,u_h\right\rangle.
 \end{equation}
 \item the term coming from the potential and from the first terms in the normal form of $a$:
 \begin{equation}\label{e:contribution-potential}
  \frac{\delta_h^2}{h^2}\left\langle\Op_h^w\left(\left\{W,\tilde{\chi}_R^B\tilde{\chi}_{\varepsilon}^C\theta_{\pm k}(hH_3)\left(aH_3+H_2X(a)-H_1X_\perp(a)\right)\right\}\right)u_h,u_h\right\rangle
 \end{equation}
 \item the remainding terms coming from the potential and the normal form of $a$:
 \begin{equation}\label{e:contribution-potential2}
  \frac{\delta_h^2}{h^2}\sum_{j=2}^4h^{j-1}\sum_{|\alpha+\beta|=j}\left\langle\Op_h^w\left(\left\{W,\tilde{\chi}_R^B\tilde{\chi}_{\varepsilon}^C\theta_{\pm k}(hH_3)\frac{Z^{\alpha}\overline{Z}^\beta}{(hH_3)^{j-1}} \mathcal{L}_{\alpha,\beta}(a)\right\}\right)u_h,u_h\right\rangle
 \end{equation}
 \item the term coming from the term of order $h^3$ in the composition rule~\eqref{e:composition-formula}:
 \begin{equation}\label{e:remainder-weyl}
  \left\langle\Op_h^w\left(\tilde{A}(D)^3\left(\left(H_1^2+H_2^2\right)(q_1,p_1)\left(\tilde{\chi}_R^B\tilde{\chi}_{\varepsilon}^C\theta_{\pm k}(hH_3)H_3\mathbf{a}\right)(q_2,p_2)\right)\right)u_h,u_h\right\rangle,
 \end{equation}
where $\tilde{A}(D):=\partial_{p_1}\cdot\partial_{q_2}-\partial_{p_2}\cdot\partial_{q_1}$.
\end{enumerate}
In other words, when $V(a)=0$, the sum of~\eqref{e:remainder-normal-form},~\eqref{e:contribution-potential},~\eqref{e:contribution-potential2} and~\eqref{e:remainder-weyl} is a $\ml{O}_{R,\varepsilon}(h^2)$ and we need to identify the main contribution.

We begin with~\eqref{e:remainder-normal-form} and~\eqref{e:contribution-potential2} which can be treated similarly. In particular, both~\eqref{e:remainder-normal-form} and~\eqref{e:contribution-potential2} are of the form
$$\left\langle\Op_h^w\left(c_hZ^\alpha\overline{Z}^{\beta}\right)u_h,u_h\right\rangle,$$
where $2\leq |\alpha+\beta|\leq 5$ and where $c_h$ is a symbol in the class $S^0_{\text{cl}}$ with seminorms uniformly bounded in terms of $0<h\leq 1$ and with support inside the one of $a\theta_{\pm k}(hH_3)\tilde{\chi}_{R}^B\tilde{\chi}_{\varepsilon}^C$. Using the composition rule~\eqref{e:composition-formula} for pseudodifferential operators, one can verify that, letting $N_0=|\alpha+\beta|$ and supposing for instance $\alpha>0$, every such term has the following expansion
\begin{multline*}\left\langle\Op_h^w\left(c_hZ^\alpha\overline{Z}^{\beta}\right)u_h,u_h\right\rangle=\left\langle\Op_h^w\left(c_hZ^{\alpha-1}\overline{Z}^{\beta}\right)\Op_h^w(Z)u_h,u_h\right\rangle\\
+\sum_{j=1}^{N_0-1}h^j\sum_{|\alpha'+\beta'|=j}\left\langle\Op_h^w\left(c_h^{\alpha',\beta'}Z^{\alpha'}\overline{Z}^{\beta'}\right)u_h,u_h\right\rangle+\ml{O}(h^{N_0}),
 \end{multline*}
where $c_h^{\alpha',\beta'}$ is a symbol in the class $S^0_{\text{cl}}$ with seminorms uniformly bounded in terms of $0<h\leq 1$ and with the same support properties as $c_h$. Iterating this argument, we end up with estimating terms that are of the form
$$\left\langle\Op_h^w\left(\tilde{c}_h\right)\Op_h^w(Z)^{\alpha}u_h,\Op_h^w(Z)^\beta u_h\right\rangle,$$
where $\tilde{c}_h$ is still a symbol in the class $S^0_{\text{cl}}$ with seminorms uniformly bounded in terms of $0<h\leq 1$ and the same support properties. We now fix a function $b(x,y,z,E)\in\ml{C}^\infty_c(\ml{U}_0\times\IR^*)$ which is identically equal to $1$ on the support $a(x,y,z)\theta_{\pm k}(E)$. Similarly, one can verify that $\tilde{\chi}_{4\varepsilon}^C$ and $\tilde{\chi}_{\frac{R}{4}}^B$ are identically equal to $1$ on the support of $\tilde{\chi}_\varepsilon^C\tilde{\chi}_R^B$. Using the composition rule~\eqref{e:composition-formula} for pseudodifferential operators and the support properties of the symbols, one finds that
\begin{multline*}\left\langle\Op_h^w\left(\tilde{c}_h\right)\Op_h^w(Z)^{\alpha}u_h,\Op_h^w(Z)^\beta u_h\right\rangle=\ml{O}_{R,\varepsilon}(h)\\
+\left\langle\Op_h^w\left(\tilde{c}_h\right)\Op_h^w\left(\tilde{c}_h^{(1)}\right)\Op_h^w(Z)^{\alpha}u_h,\Op_h^w\left(\tilde{c}_h^{(1)}\right)\Op_h^w(Z)^\beta u_h\right\rangle,
 \end{multline*}
where $\tilde{c}_h^{(1)}:=\tilde{\chi}_{\frac{R}{4}}^B\tilde{\chi}_{4\varepsilon}^Cb(.,hH_3)$. Using~\eqref{e:aprioribound-subtle}, we can finally conclude that both~\eqref{e:remainder-normal-form} and~\eqref{e:contribution-potential2} are $\ml{O}_{R,\varepsilon}(h)+\ml{O}_{R,\varepsilon}(\delta_h^2h^{-1})$. Hence,
\begin{equation}\label{e:intermediary-bound}
 \eqref{e:contribution-potential}+\eqref{e:remainder-weyl}=\ml{O}_{R,\varepsilon}(h)+\ml{O}_{R,\varepsilon}(\delta_h^2h^{-1}).
\end{equation}

We now turn to the estimate on~\eqref{e:remainder-weyl} which will also turn out to be negligible. Thanks to the composition rule of Theorem~\ref{t:composition}, we have that the symbol appearing in~\eqref{e:remainder-weyl} is in the class $S^{0}_{\text{cl}}$ (with uniform bounds in terms of $0<h\leq 1$ on the seminorms). In particular, it is bounded by some uniform constant $C$. Thanks to the Calder\'on-Vaillancourt Theorem, it leads to a crude upper bound of order $\ml{O}_{R,\varepsilon}(1)$. Combining this observation with~\eqref{e:intermediary-bound}, we find that
\begin{multline}\label{e:main-technical-estimate}
 \left\langle\Op_h^w\left(\left\{W,\tilde{\chi}_R^B\tilde{\chi}_{\varepsilon}^C\theta_{\pm k}(hH_3)\left(aH_3+H_2X(a)-H_1X_\perp(a)\right)\right\}\right)u_h,u_h\right\rangle\\
 =\ml{O}_{R,\varepsilon}(h^2\delta_h^{-2})+\ml{O}_{R,\varepsilon}(h).
\end{multline}

\begin{remark}
We note that this is exactly when analyzing the contribution of the term~\eqref{e:remainder-weyl} that we miss the critical case $\delta_h=h$ due to the above crude estimate $\ml{O}_{R,\varepsilon}(1)$. We refer to the end of the proof for a more detailed discussion.
 \end{remark}

We now remove the cutoffs from the Poisson bracket up to some small remainder terms on the left hand-side of~\eqref{e:main-technical-estimate}. We first notice that, if one derivative hits $\tilde{\chi}_R^B$, then the resulting term is identically $0$ for $h$ small enough (depending on $R$). Hence, we only need to understand the property of the terms that appear when one differentiates $\tilde{\chi}_{\varepsilon}^C\theta_{\pm k}(hH_3)\left(aH_3+H_2X(a)-H_1X_\perp(a)\right)$. Similarly, if we differentiate $\tilde{\chi}_\varepsilon^C$, then, we end up with a term in $S^{0}_{\text{cl}}$ such that, on its support, 
$$\sqrt{1+H_1^2+H_2^2}\leq \varepsilon|H_3|\leq2\sqrt{1+H_1^2+H_2^2}.$$
Using the support properties of such a term, we can multiply it $(1+H_1^2+H_2^2)^{-1}$ and we get a symbol in $S^{-2}_{\text{cl}}$ (thus $\lesssim h^2$ thanks to the support properties). Using the semiclassical a priori estimate~\eqref{e:apriori-semiclassical} together with the composition rule~\eqref{e:composition-formula}, we find that any such term will give an upper bound of order $\ml{O}_{R,\varepsilon}(h^2)$, hence negligibe compared with the estimate~\eqref{e:main-technical-estimate}. Finally, if one differentiates $\theta_{\pm k}(hH_3)$, it leads to negligible terms as $\theta_{\pm k}'(E)$ is equal to $0$ on the support of $\mu_{k,\infty}^\pm$. Thus, equality~\eqref{e:main-technical-estimate} becomes
\begin{multline}\label{e:main-technical-estimate2}
 \left\langle\Op_h^w\left(\tilde{\chi}_R^B\tilde{\chi}_{\varepsilon}^C\theta_{\pm k}(hH_3)\left\{W,aH_3+H_2X(a)-H_1X_\perp(a)\right\}\right)u_h,u_h\right\rangle\\=\ml{O}(h^2\delta_h^{-2})+r(h,R,\varepsilon),
\end{multline}
where $r(h,R,\varepsilon)$ verifies
\begin{equation}\label{e:bound-remainder}
 \limsup_{\varepsilon\to 0^+}\limsup_{R\to\infty}\limsup_{h\to 0^+}r(h,R,\varepsilon)=0,
\end{equation} 
Letting $h\to 0^+$, $R\to\infty$ and $\varepsilon\to 0^+$ (in this order), one finds that
$$\nu_{k,\infty}^\pm\left(aV(W)+X(a)X_\perp(W)-X_\perp(a)X(W)\right)=0.$$
Using the invariance of $\nu_{k,\infty}^\pm$ together with the fact that $V(a)=0$, we finally obtain the expected result.



\subsection{The case $\delta_h=h$}\label{r:term-order-3}

Let us now explain what would need to be done to deal with the case $\delta_h=h$. As already explained, it requires to analyze more precisely the terms coming from the remainder~\eqref{e:remainder-weyl} of order $h^3$ in the composition formula for the Weyl quantization. The two involved symbols in this term are $|Z|^2=H_1^2+H_2^2$ (which belongs to $S^{2}_{\text{cl}}$) and
$$b=H_3\sum_{\alpha+\beta\leq 4}\tilde{\ml{L}}_{\alpha,\beta}(a)(hZ)^{\alpha}(h\overline{Z})^{\beta} \frac{\theta_{\pm k}(hH_3)}{(hH_3)^{\alpha+\beta}}\tilde{\chi}_R^B\tilde{\chi}_{\varepsilon}^C.$$
Hence, we need to understand the properties of symbols that are of the form
\begin{equation}\label{e:term-of-interest-weyl}\tilde{A}(D)^3\left(|Z|^2(q_1,p_1)\left(H_3a_{\alpha,\beta}(hZ)^{\alpha}(h\overline{Z})^{\beta} \frac{\theta_{\pm k}(hH_3)}{(hH_3)^{\alpha+\beta}}\tilde{\chi}_R^B\tilde{\chi}_{\varepsilon}^C\right)(q_2,p_2)\right),
\end{equation}
when evaluated at $(q_1,p_1)=(q_2,p_2)$ and where $a_{\alpha,\beta}\in\ml{C}^\infty_c(\ml{U}_0)$ with $0\leq \alpha+\beta\leq 4$. As when dealing with~\eqref{e:main-technical-estimate}, we can remove the cutoff functions from this term (up to small remainders) and we are left with analyzing
\begin{equation}\label{e:term-of-interest-weyl2}\tilde{\chi}_R^B\tilde{\chi}_{\varepsilon}^C\tilde{A}(D)^3\left(|Z|^2(q_1,p_1)\left(a_{\alpha,\beta}\left(\frac{Z}{H_3}\right)^{\alpha}\left(\frac{\overline{Z}}{H_3}\right)^{\beta}H_3 \theta_{\pm k}(hH_3)\right)(q_2,p_2)\right),
\end{equation}
when evaluated at $(q_1,p_1)=(q_2,p_2)$ and when one differentiates at most twice with respect to $p_1$ (as $|Z|^2$ is quadratic in the $p$ variable).

We can now make use of the simplified expression~\eqref{e:expression-hamiltonien} for the Hamiltonian $|Z|^2$ which shows that it does depend on the $z$ variable. Hence differentiating $H_3$ with respect to the $p$ variable will only result into zero contributions. In other words, we can restrict ourselves to the case $1\leq \alpha+\beta\leq 4$ and to the case where the derivatives with respect to $q_1=(x_1,y_1,z_1)$ only involves derivatives with respect to the variables $(x_1,y_1)$ and~\eqref{e:term-of-interest-weyl2} becomes
\begin{equation}\label{e:term-of-interest-weyl4}\frac{\tilde{\chi}_R^B\tilde{\chi}_{\varepsilon}^C}{H_3^{\alpha+\beta-1}}\theta_{\pm k}(hH_3)\left(\partial_{p_1}\cdot\partial_{q_2}-\partial_{p_2}\cdot\partial_{q_1}\right)^3\left(|Z|^2(q_1,p_1)\left(a_{\alpha,\beta}Z^{\alpha}\overline{Z}^{\beta}\right)(q_2,p_2)\right).
\end{equation}
Moreover, the contribution coming from the terms with $\alpha+\beta=4$, will involve symbols of the form $H_j r/H_3$ with $r\in S^{0}_{\text{cl}}$ and $j\in\{1,2\}$. Thanks to the localization of the cutoff functions and to the Calder\'on-Vaillancourt Theorem, it yields a contribution $\ml{O}(\varepsilon)+\ml{O}_{R,\varepsilon}(h^{1/2})$. Thus, we would only need to discuss the cases $1\leq \alpha+\beta\leq 3$ to deal with the case $\delta_h=h$. 
 In order to conclude, one would need to analyze the precise form of~\eqref{e:term-of-interest-weyl4} for $1\leq \alpha+\beta\leq 3$ (and thus more precisely the terms appearing in the normal form procedure). In the present state of the analysis, it is not transparent if these terms will sum up to $0$. Computing these terms is a little bit involved and we do not pursue this here.

\section{Proof of the main Theorem}\label{s:endproof}
Thanks to Proposition~\ref{p:mainprop}, the proof of Theorem~\ref{t:maintheo} is almost complete. Indeed, the conclusion of this Proposition can be equivalently rewritten as
$$\forall k\geq 0,\quad\nu_{k,\infty}^\pm\left(\mathbf{A}_+(W)\mathbf{A}_-(a)-\mathbf{A}_-(W)\mathbf{A}_+(a)\right)=0,$$
where $\mathbf{A}_\pm=X\pm iX_\perp$ and where $a$ is a function verifying $V(a)=0$. We now decompose $W$ according to~\eqref{e:fourier-decomposition} and we find
$$\forall k\geq 0,\quad\sum_{n\in\IZ}\nu_{k,\infty}^\pm\left(\mathbf{A}_+(\widehat{W}_n)\mathbf{A}_-(a)-\mathbf{A}_-(\widehat{W}_n)\mathbf{A}_+(a)\right)=0.$$
Using the invariance by the flow and the facts that the functions $\mathbf{A}_+(\widehat{W}_n)\mathbf{A}_-(a)$ and $\mathbf{A}_-(\widehat{W}_n)\mathbf{A}_+(a)$ both belong to $\ml{C}^\infty_n(\ml{M})$ (as $V(a)=0$), this sum reduces to
$$\forall k\geq 0,\quad\nu_{k,\infty}^\pm\left(\mathbf{A}_+(\widehat{W}_0)\mathbf{A}_-(a)-\mathbf{A}_-(\widehat{W}_0)\mathbf{A}_+(a)\right)=0.$$
This is true for any function $a$ verifying $V(a)=0$. If we now consider a general $a\in\ml{C}^\infty(\ml{M})$, we find
$$\forall k\geq 0,\quad\nu_{k,\infty}^\pm\left(\mathbf{A}_+(\widehat{W}_0)\mathbf{A}_-(\widehat{a}_0)-\mathbf{A}_-(\widehat{W}_0)\mathbf{A}_+(\widehat{a}_0)\right)=0.$$
Applying the exact same argument that we used with $W$ (but with $a$), it ends the proof of Theorem~\ref{t:maintheo}.



\appendix

\section{The variable curvature case and the case $\delta_h=1$}\label{s:variable}

In this appendix, we briefly discuss the case where $(M,g)$ has (variable) nonvanishing curvature and where $\delta_h=1$. For simplicity, we also make the assumption that the constant $E_0$ appearing in~\eqref{e:eigenvalue-equation} is $>\max W$. In that case, the main result from~\cite{ArnaizRiviere23} adapted to our setting\footnote{Recall that we have just intertwined the roles of $X$ and $V$ compared with that reference.} shows that the measures $(\nu_{k,\infty}^\pm)_{k\geq 0}$ and $\overline{\nu}_{\infty}$ are invariant by flow
$$Y_W:=KV+X(K-\ln(E_0-W))X_\perp-X_\perp(K-\ln(E_0-W))X.$$
If we make the extra-assumption that $W$ is the pullback of a function on $M$ (i.e. independent of $z$) as $K$ is, then $Y_W(K-\ln(E_0-W))=0$.
In particular, each orbit of the flow is contained in a connected component of the level sets
$$\ml{E}_{K_0}:=\{(m,v)\in SM: K(m)=\ln(E_0-W(m))+K_0\}.$$
As a consequence, any invariant probability measure of the vector $Y_W$ is a convex combination of the invariant probability measures that are supported inside a connected component $\mathcal{C}$ of some level set $\ml{E}_{K_0}$. We also emphasize that, for $W$ independent of $v$, the first part of Corollary~\ref{c:maincoro} always holds true with $K-\ln(E_0-W)$ replacing $\widehat{W}_0$, i.e. for all $k\geq 0$ and for all $a\in\ml{C}^1(M)$,
\begin{equation}\label{e:variable-curvature}
\int_M\nabla_g^\perp (K-\ln(E_0-W))(a)d\left(\Pi_*\nu_{k,\infty}^\pm\right)=0.
\end{equation}
In fact, this also remains true for $\overline{\nu}_\infty$ in that setting.

Let us conclude with describing the allowed invariant measures on $SM$ in the case where $W\equiv 0$. In that setting, one has in fact $[X(K)X_\perp-X_\perp(K)X,KV]=0.
$ From this, we have two options regarding invariant measures:
\begin{itemize}
 \item \textbf{The connected component $\ml{C}$ of $\{K(m)=K_0\}$ contains one (or several) critical point of $K$}. In that case, invariant probability measures are convex combinations of
 $$\left\{\delta_{m_0}(m)\otimes \frac{dv}{(2\pi K_0)}:m_0\in\text{Crit}(K)\cap \ml{C}\right\}.$$
 \item \textbf{The connected component $\ml{C}$ contains no critical point of $K$}. In that case, the flow induced by the vector field $\nabla_g^\perp (K)$ on $\Pi(\ml{C})$ is periodic of minimal period $T(\ml{C})>0$. As the vector fields $KV$ and $X(K)X_\perp-X_\perp(K)X$ commute, the orbits are either all periodic, or all dense in the torus $\ml{C}$. The second case occurs if and only if $T(\ml{C}) K_0\notin\pi\mathbb{Q}$. In that case, the only invariant measure is the Lebesgue measures induced by these two vector fields on $\ml{C}$. In the periodic case, all points have the same period given by $nT(\ml{C})K_0$ with $n$ the smallest positive integer such that $n T(\ml{C})K_0\in 2\pi\IZ_+$. Note that there does not seem to be a reason why one of the two situations does not occur for a generic metric $g$.
\end{itemize}

\section{Invariance properties of $\nu_{\text{comp}}$}

\label{a:compact}

In this appendix, we complete the discussion from~\S\ref{ss:compact} by briefly describing the invariance properties of the measure $\nu_{\text{comp}}$. In particular, we connect them with the magnetic Hamiltonian on $T^*M$ induced by the metric $g$. Recall from~\S\ref{ss:compact} that this part of the limit measure is the projection on $\ml{M}$ of a (standard) semiclassical measure which is invariant under the Hamiltonian flow on $T^*\ml{M}$ associated with $H_1^2+H_2^2$ and restricted to the energy layer $\{H_1^2+H_2^2=E_0\}$, with $E_0>0$. In the local isothermal coordinates of Section~\ref{s:chart}, one has
\begin{equation}\label{e:expression-hamiltonien}
H_1^2+H_2^2=e^{-2\lambda}\left((\xi+\partial_y\lambda \zeta)^2+(\eta-\partial_x\lambda \zeta)^2\right).
 \end{equation}
Writing down the Hamilton-Jacobi equation, one finds that $\zeta$ is constant along the trajectories of the flow. Hence, one can fix the value of $\zeta$ to be equal to some $B\in\IR$ and describe the evolution for a fixed $B$. It induces an Hamiltonian flow on $T^*\IR^2$ whose Hamiltonian function is given by
$$H_B(x,y,\xi,\eta):=e^{-2\lambda}\left((\xi+B\partial_y\lambda)^2+(\eta-B\partial_x\lambda )^2\right).$$
The Hamiltonian $H_B$ can be identified with a magnetic Hamiltonian for the metric $g$ and the magnetic potential $\Theta=B\left(\partial_y\lambda dx-\partial x\lambda dy\right).$ The corresponding magnetic field is then given by
$$d\Theta=-B\left(\partial_x^2\lambda+\partial_y^2\lambda\right)dx\wedge dy=BK(x,y)e^{2\lambda(x,y)}dx\wedge dy=BK(x,y)\text{Vol}_g(dx,dy).$$

Finally, as $H_B$ is preserved under the action of the Hamiltonian flow, it is natural to set $(\cos z_1,\sin z_1)=\frac{e^{-\lambda}}{\sqrt{E_0}}(\xi+B\partial_y\lambda,\eta-B\partial_x\lambda)$ and the Hamilton-Jacobi equation on the energy layer $\{H_B=E_0\}$ becomes (after simplifications)
$$x'=2\sqrt{E_0}e^{-\lambda}\cos z_1,\quad y'(t)=2\sqrt{E_0}e^{-\lambda}\sin z_1,$$
and
$$z_1'=2\sqrt{E_0}e^{-\lambda}(-\sin z_1 \partial_x\lambda+\cos z_1\partial_y\lambda)+2BK=z'+2KB.$$
Hence, up to multiplication by $1/2\sqrt{E_0}$, we recognize the vector field $X+\frac{BK}{\sqrt{E_0}}V$ (in the coordinates $(x,y,z_1)$ adapted to the geometry of the magnetic Hamiltonian).


\begin{remark}\label{r:magnetic-zelditch} When considering magnetic eigenfunctions as in~\S\ref{ss:compact} with $B_h\rightarrow B\in\IR$, i.e. joint solutions to~\eqref{e:eigenvalue-equation} and to~\eqref{e:magnetic-eigenvalues}, the corresponding semiclassical measure on $T^*\ml{M}$ are carried by the compact set $\{H_1^2+H_2^2=E_0\}\cap \{H_3=B\}$ and invariant by the Hamiltonian flows of $H_1^2+H_2^2$ and $H_3$ (which acts like $V$ along the $z$ variable). 
\end{remark}

\begin{remark} Note that the calculations that we made so far in this appendix are independent of the fact that the curvature is constant. 
\end{remark}

The regime $|H_3|=|B|\rightarrow\infty$ that we consider in this article is in some sense the opposite of the setting considered by Zelditch in~\cite[\S3.20]{Zelditch92b} (for $K\equiv -1$ and $W\equiv 0$). Indeed, in this reference, following earlier works of Schrader and Taylor~\cite{SchraderTaylor89}, he proved a quantum ergodicity property for magnetic eigenfunctions in the regime where $|B|$ is finite with our conventions (as in Remark~\ref{r:magnetic-zelditch}). Recall that, in this reference, he rather considered the elliptic eigenvalue equation
\begin{equation}\label{e:KK-eigenvalue}
-h^2(X^2+X_\perp^2+V^2)\psi_h=\tilde{E}_h\psi_h,\quad\|\psi_h\|_{L^2}=1,\quad \tilde{E}_h\rightarrow \tilde{E}_0\in\IR_+^*\quad\text{as}\quad h\rightarrow 0^+
\end{equation}
on $\Gamma\backslash PSL(2,\IR)$. This is the so-called Kaluza-Klein Laplacian. In particular, his main operator is elliptic rather than hypoelliptic as our operator $\Delta_{\text{sR}}$. This extra property modifies the oscillation properties of eigenfunctions which now oscillate like $h^{-1}$ and not in a range lying between $h^{-1}$ and $h^{-2}$. In particular, the standard theory of semiclassical measures applies and any such measure is a \emph{probability}\footnote{Due to ellipticity, there is now no escape of mass at infinity.} measure carried by the compact set $\{H_1^2+H_2^2+H_3^2=\tilde{E}_0\}$ and invariant by the Hamiltonian flow of $H_1^2+H_2^2+H_3^2$. In~\cite{Zelditch92b}, Zelditch considered in addition that eigenfunctions verifies~\eqref{e:magnetic-eigenvalues} (as $V$ also commutes with his operator) which leads for our hypoelliptic operator to the eigenvalue equation:
\begin{equation}\label{e:KK-hypoelliptic}-h^2\Delta_{\text{sR}}\psi_h=(\tilde{E}_h-B_h^2)\psi_h.\end{equation}
Using the tools on semiclassical measures from~\cite[Ch.~5]{Zworski12}, such joint eigenfunctions result into measures that are also carried on the set $\{H_3=B\}$ and invariant by the flow generated by $H_3$. The main focus in~\cite{Zelditch92b} is on the case $B^2<\tilde{E}_0$ where Zelditch proves a quantum ergodicity Theorem provided $B^2\leq \tilde{E}_0/2$.

\begin{remark}One can deduce from~\eqref{e:KK-hypoelliptic} that $E_h=\tilde{E}_h-B_h^2\geq 0$ (thus $B_h=\ml{O}(1)$). Recall that, all along the article, we made the assumption that the eigenvalue $E_h$ does not tend to $0$ in view of enhancing the hypoelliptic behaviour of eigenfunctions in the high frequency regime. With the conventions of~\eqref{e:KK-eigenvalue}, the measures $\overline{\nu}_\infty$ and $(\nu_{k,\infty}^\pm)_{k\geq 0}$ would thus arise by picking $\tilde{E}_0=B^2$, and they would correspond to the part of the semiclassical measure carried on $T^*\ml{M}$ by $\{H_1^2+H_2^2=0\}\cap\{H_3=\pm\sqrt{\tilde{E}}_0\}$ with the choice of semiclassical scaling from~\eqref{e:KK-eigenvalue}. In order to recover the semiclassical scaling of the introduction, one would need to replace $h$ by $\tilde{h}:=\frac{h}{\sqrt{\tilde{E}_h-B_h^2}}$ in~\eqref{e:KK-hypoelliptic} when $\tilde{E}_0=B^2$.
 \end{remark}



\section{Pseudodifferential calculus on $\IR^2\times\IS^1$}
\label{a:pdo}

In this appendix, we review a few facts about semiclassical analysis on $T^*(\IR^2\times\IS^1)$ that are used all along our analysis of the measure at infinity. A standard textbook is~\cite{Zworski12} which treats the case of $T^*\IR^3$ in great details in Chapter~$4$. The case of $T^*(\IR^2\times\IS^1)$ can be handled similarly by proper use of Fourier series along the $z$-variable rather than Fourier transform. See for instance~\cite[\S 5.3]{Zworski12} for a detailed discussion in the case of $T^*\IT^3$.

For a nice enough smooth function $a$ on $T^*(\IR^2\times\IS^1)$ (say compactly supported) and for every $h>0$, the Weyl (semiclassical) quantization of $a$ is defined, for all $u$ in $u\in\ml{C}^\infty_c(\IR^3)$, by
\begin{equation}\label{e:weyl} \text{Op}_h^w(a)\left(u\right)(q):=\frac{1}{(2\pi h)^3}\int_{\IR^6}e^{\frac{i}{h}(q-q')\cdot p}a\left( \frac{q+q'}{2},p\right)u(q')dq'dp.
\end{equation}
Using the periodicity along the $\mathbb{S}^1$-variable, one can verify that this definition extends to smooth test functions $u\in\ml{C}^\infty_c(\IR^2\times\IS^1)$~\cite[\S5.3.1]{Zworski12}.

Regarding the regularity needed for $a$, this definition still makes sense when working with smooth functions $a$ belonging to the class of (Kohn-Nirenberg) symbols~\cite[\S 9.3]{Zworski12}:
$$S^m_{\text{cl}}(T^*(\IR^2\times\IS^1))=\left\{a\in\ml{C}^\infty(T^*(\IR^2\times\IS^1)):\ \forall(\alpha,\beta)\in\IZ_+^{6},\ P_{m,\alpha,\beta}(a)<+\infty\right\},$$
where $m\in\IR$ and
$$P_{m,\alpha,\beta}(a):=\sup_{(q,p)}\{\langle p\rangle^{-m+|\beta|}|\partial_q^\alpha\partial_p^\beta a(x,\xi)|\}.$$
In other words, we gain some decay in $p$ when differentiating along the $p$-variable. Even if such a decay is not necessary to work in an Euclidean set-up, it is of crucial importance in our analysis to have this extra decay in view of dealing with the escape at infinity in the fibers. This class of symbols is often denoted by $S^{m}_{\text{KN}}$ in microlocal set-ups where one wants to distinguish the smaller class of homogeneous symbols. 

A nice property of the Weyl quantization is that, for a real-valued $a$, $\Op_h^w(a)$ is a (formally) selfadjoint operator~\cite[Th.~4.1]{Zworski12}. Another property that we extensively use all along this article is the composition rule for pseudodifferential operators\footnote{Technically speaking, this reference deals with the Weyl quantization on $T^*\IR^3$ but the proof works as well in our set-up.}~\cite[Th.~9.5, Th.~4.12]{Zworski12}
\begin{theorem}\label{t:composition} Let $a\in S^{m_1}_{\operatorname{cl}}(T^*(\IR^2\times\IS^1))$ and $b\in S^{m_2}_{\operatorname{cl}}(T^*(\IR^2\times\IS^1))$. Then, there exists $c\in S^{m_1+m_2}_{\text{cl}}(T^*(\IR^2\times\IS^1))$ (depending on $h$) such that
\begin{equation}\label{e:composition-formula}\Op_h^w(a)\circ\Op_h^w(b)=\Op_h^w(c).\end{equation}
Moreover,
$$c(q,p)=\sum_{k=0}^N\frac{h^k}{k!}\left(A(D)\right)^k(a(q_1,p_1)b(q_2,p_2))|_{q_1=q_2=q, p_1=p_2=p}+\ml{O}_{S^{m_1+m_2-N-1}}(h^{N+1}),$$
 where the constant in the remainder depends on a finite number of seminorms of $a$ and $b$ (depending on $N$ and on the seminorm in $S^{m_1+m_2-N-1}_{\operatorname{cl}}$), and where 
 $$A(D):=\frac{1}{2i}\left(\partial_{p_1}\cdot\partial_{q_2}-\partial_{p_2}\cdot\partial_{q_1}\right).$$
\end{theorem}
In particular, we can see from this result that $c=\ml{O}_{S^{m_1+m_2-N-1}}(h^{N+1})$ if $a$ and $b$ have disjoint supports. We can also verify that, all the even powers in $h$ in the asymptotic expansion of $[\Op_h^w(a),\Op_h^w(b)]$ cancels out and that the first term is given by $\frac{h}{i}\{a,b\}.$ Another key property for us is the Calder\'on-Vaillancourt Theorem~\cite[Ch.~5]{Zworski12} that states the existence of constants $C_0, N_0$ such that, for every $a\in S^{0}_{\text{cl}}(T^*(\IR^2\times\IS^1))$,
\begin{equation}\label{e:calderon}
 \left\|\Op_h^w(a)\right\|_{L^2\rightarrow L^2}\leq C_0\sum_{|\alpha|\leq N_0}h^{\frac{|\alpha|}{2}}\|\partial^\alpha a\|_{\infty}.
\end{equation}

\bibliographystyle{amsalpha}

\end{document}